\documentclass[a4paper,12pt]{article}

\usepackage[
    a4paper,
    left = 20mm,
    right = 20mm,
    top = 25mm,
    bottom = 25mm,
    bindingoffset = 0cm
]{geometry}

\usepackage{amsmath,amssymb,amsthm}
\usepackage{xypic}

\usepackage[
    draft = false,
    pdftex = true,
    unicode = true,
    linkcolor = blue,
    urlcolor = blue,
    citecolor = blue,
    raiselinks = true,
    colorlinks = true,
    hyperfootnotes = true,
    pdfauthor = {Alex\ Dainiak},
    pdftitle = {Inverse\ problems\ for\ the\ number\ of\ maximal\ independent\ sets}
]{hyperref}

\newtheorem{theorem}{Theorem}
\newtheorem{lemma}{Lemma}
\newtheorem*{conjecture}{Conjecture}
\newtheorem*{remark}{Remark}

\author{Alex Dainiak}
\title{Inverse problems for the number\\ of maximal independent sets}
\date{\today}

\begin{document}
\maketitle

\begin{abstract}
We study the following inverse graph-theoretic problem: how many vertices should a graph have given that it has a specified value of some parameter. We obtain asymptotic for the minimal number of vertices of the graph with the given number $n$ of maximal independent sets for a class of natural numbers that can be represented as concatenation of periodic binary words.
\end{abstract}

\par Problems of estimating various graph invariants play the central role in quantitative graph theory. Among the most studied invariants are connectivity, chromatic number, girth, independence number, maximal clique size, number of independent sets etc. As well as forward problems, inverse problems also are of interest. They generally can be stated as follows: find a graph (or prove its existence) that have the desired value of some parameter. The classical problem of this kind is finding a graph with the given degree sequence~\cite{erdosgallai,havel}. For a long time it was not known if there was only a finite number of naturals not being the Wiener index of trees.~\cite{wagner,wang}. An analogous question considering the number of independent sets in trees, asked in~\cite{linek}, is not yet solved, whereas some other parameters of trees are better studied (e.g.~\cite{czabarka}).

\par We now state the problems coevered in this paper in their general form. Let $\mathcal{G}$ be a family of graphs, and let~$\phi:\:\mathcal{G}\to S$ and $\psi:\:\mathcal{G}\to T$ be arbitrary functionals on $\mathcal{G}$. The \emph{existential inverse problem} for the pair~$(\mathcal{G},\,\phi)$ may be stated as follows: ``describe all $s\in S$ for which there exists a graph $G\in\mathcal{G}$ having $\phi(G)=s$''.

\par Let $S$ be the set of all values of $\psi$ for all graphs in $\mathcal{G}$. For $S\subseteq{\mathbb N}$ we call~$\mathcal{G}$ to be \emph{strongly $\phi$-complete}, if for every $s\in S$ there is $G\in\mathcal{G}$ such that $\phi(G)=s$. If such $G\in\mathcal{G}$ exists for all large enough $s\in S$, then we say that $\mathcal{G}$ is \emph{weakly $\phi$-complete}, or just \emph{$\phi$-complete}. If $\phi(G)=s$ then we say that~$s$ is \emph{realized} by~$G$.

\par If the existential inverse problem is solved positively, we can consider the \emph{optimizational inverse problem} for the triple~$(\mathcal{G},\,\phi,\,\psi)$: ``for a given~$s\in S$ find $L_{\phi,\,\psi}^\mathcal{G}(s)=\inf\{\psi(G)\mid G\in\mathcal{G},\,\phi(G)=s\}$''. As the problem of finding $L$ exactly is too hard, it is natural to consider only the asymptotic behavior of $L_{\phi,\,\psi}^\mathcal{G}(s)$ for $\phi$-complete families of graphs. If $\mathcal{G}$ is a class of all graphs, we shorten the notation $L_{\phi,\,\psi}^\mathcal{G}(s)$ to $L_{\phi,\,\psi}(s)$.

\par Denote by~$\iota(G)$ the number of all independent sets (i.\,s.) of vertices in~$G$, and by $\iota_m(G)$ the number of maximal-by-inclusion i.s. (m.\,i.\,s.) in~$G$. Finally, by $\iota_M(G)$ we denote the number of maximum independent sets in $G$. We write~$\nu(G)$ and~$\epsilon(G)$ for the number of vertices and edges in~$G$ respectively. The families of bipartite graphs and forests are denoted by $\mathcal{B}$ and $\mathcal{F}$ respectively. We write $K_r$ and $P_r$ for complete graphs and paths on $r$ vertices. $K_{r,s}$ denotes complete bipartite graph, $r$ and $s$ being the sizes of its parts. $K_{r,r}'$ stands for the \emph{corona-graph}, which can be constructed by deleting edges of some perfect matching from $K_{r,r}$. The sets of vertices and edges of graph $G$ will be denoted as $V_G$ and $E_G$. The edge between vertices $u$ and $v$ will be denoted as $uv$. A family of all m.\,i.\,s. of $G$ is denoted as $\mathcal{I}_m(G)$.

\par In the notation of the current paper, Linek proved~\cite{linek} the strong $\iota$-completeness of~$\mathcal{B}$. It justifies the consideration of the optimization problem of finding $L_{\iota,\,\nu}^{\mathcal{B}}(n)$. The trivial lower bound is $L_{\iota,\,\nu}^{\mathcal{B}}(n)\geqslant\log_2n$ (which follows from the inequation $\iota(G)\le2^{\nu(G)}$). A graph constructed in~\cite{linek} to realize a given natural number has the maximal possible sizes of parts: $\lfloor\log_2n\rfloor$ and $\lfloor\log_2(n-2^{\lfloor\log_2n\rfloor}+1)\rfloor$. For $n=2^k-1$ such graph would have $2k-2$ vertices, which is double the expected optimal size. Some $n$ of the above form can be realized more economically, as the following statement shows.

\begin{theorem}
For $k=2^t$ we have $L_{\iota,\,\nu}^{\mathcal{F}}(2^k-1)\lesssim k$.
\end{theorem}

\begin{proof}
Just note that $$2^{2^t}-1=\prod_{j=0}^{t-1}(2^{2^j}+1)=\prod_{j=0}^{t-1}\iota(K_{2^j,1})=\iota\left(\bigsqcup_{j=0}^{t-1}K_{2^j,1}\right).$$ At the same time $\nu(\bigsqcup_{j=0}^{t-1}K_{2^j,1})=2^t+t-1\lesssim k$.
\end{proof}

The existential inverse problem for~$(\mathcal{B},\,\iota_m)$ is trivial, as any $n\ge4$ can be realized as the number of m.\,i.\,s. in corona-graph $K_{n-2,n-2}'$. If we consider $\psi$ to be the number of vertices of a graph, we come up with an optimizational inverse problem: ``for natural $n$ find minimal $L(n)$ such that there exists a graph on $L(n)$ vertices having $n$ maximal independent sets''. The remaining part of the paper is dedicated to estimating $L(n)$.

\section*{Bounds for $L_{\iota_m,\nu}^{\mathcal{B}}(n)$}
\begin{lemma} \label{lem_recalc_iotamgen}
Let $G$ be bipartite with parts $L_G,\,R_G$ and without isolated vertices. Let $\widetilde{G}$ be a bipartite graph, vertex-disjoint with $G$, and let $U_1$ and~$U_2$ be some subsets of first and second parts of $\widetilde{G}$ respectively. Let $G'$ be a graph obtained by connecting all vertices in $U_1$ (resp. $U_2$) to all vertices in $L_G$ (resp. $R_G$). Then we have
$$\iota_m(G')=(\iota_m(G)-2)\cdot\iota_m(\widetilde{G}\setminus(U_1\cup U_2))+\iota_m(\widetilde{G}\setminus U_1)+\iota_m(\widetilde{G}\setminus U_2)+\iota_m(\widetilde{G}+U_1+U_2),$$
where $\iota_m(\widetilde{G}+U_1+U_2)$ stands for the number of m.\,i.\,s. of $\widetilde{G}$ having non-empty intersections with both $U_1$ and $U_2$.
\end{lemma}
\begin{proof}
The statement of the lemma can be checked by direct counting. If an m.\,i.\,s. of $G'$ contains no vetices of~$G$, then it must contain at least one vertex from both $U_1$ and~$U_2$, the number of such sets being~$\iota_m(\widetilde{G}+U_1+U_2)$. If a m.\,i.\,s. of~$G'$ contains vertices from both parts of~$G$, then it is disjoint with $U_1\cup U_2$, and its subsets in $G$ and $\widetilde{G}$ must themselves be maximal independent sets in $G$ and $\widetilde{G}$ respectively. Thus the number of such m.\,i.\,s. equals $(\iota_m(G)-2)\cdot\iota_m(\widetilde{G}\setminus(U_1\cup U_2))$. If a m.\,i.\,s. of $G$ contains all vertices of $L_G$ or whole $R_G$, then its subset in $\widetilde{G}$ will form a m.\,i.\,s. in $\widetilde{G}\setminus U_1$ or $\widetilde{G}\setminus U_2$ respectively.
\end{proof}

Let $\widetilde{G}$ be bipartite with $U_1$ and $U_2$ being some subsets of its parts. Put $$h_{\widetilde{G}}'=\iota_m(\widetilde{G}\setminus(U_1\cup U_2)),$$ $$h_{\widetilde{G}}''=(\iota_m(\widetilde{G}\setminus U_1)+\iota_m(\widetilde{G}\setminus U_2)+\iota_m(\widetilde{G}+U_1+U_2)-2\iota_m(\widetilde{G}\setminus(U_1\cup U_2))).$$

\begin{lemma}\label{lem_main_iotamgen}
Let $\Gamma$ be a finite set of bipartite graphs with selected subsets in their parts, such that $\{h_{\widetilde{G}}'k+h_{\widetilde{G}}''\mid k\in{\mathbb N},\,\widetilde{G}\in\Gamma\}\supseteq({\mathbb N}\setminus [1,n_0])$ for some $n_0$. Put
$$\gamma=\max\left\{(\log_2h_{\widetilde{G}}')^{-1}\nu(\widetilde{G})\mid \widetilde{G}\in\Gamma\right\}.$$
Then $L_{\iota_m,\,\nu}^{\mathcal{B}}(n)\le \gamma\cdot \log_2 n + O(1)$.
\end{lemma}

\begin{proof}
The lemma is proved by induction on~$n$ with a help of lemma~\ref{lem_recalc_iotamgen}. Let $\Gamma_0$ be an arbitrary finite set of bipartite graphs having $\{\iota_m(G)\mid G\in\Gamma_0\}\supseteq[1,n_0]$. For example, as $\Gamma_0$ we can take the set $\{K_{n-2,n-2}'\mid n\in [4,n_0]\}\cup\{K_1,\,K_{1,1},\,P_4\}$. Let $\nu_0$ be the maximal number of vertices of graphs in $\Gamma_0$. It suffices to prove that for any $n$ the following inequality holds:
\begin{equation}\label{eq_iotamproof1}
L_{\iota_m,\,\nu}^{\mathcal{B}}(n)\le \gamma\cdot \log_2 n+\nu_0,
\end{equation}
which would imply the statement of the lemma.
\par The inequality \eqref{eq_iotamproof1} trivially holds for $n\le n_0$. Consider an arbitrary $n',\,n'>n_0$, and assume that \eqref{eq_iotamproof1}~holds for all $n$ less than $n'$. By the conditions of the lemma, there exists some $\widetilde{G}\in\Gamma$ and some natural~$k$, such that $n'=h_{\widetilde{G}}'k+h_{\widetilde{G}}''$. By the induction hypothesis there is some bipartite $G$ having $\iota_m(G)=k$ and $\nu(G)\le\gamma\cdot \log_2 k+\nu_0$. By lemma~\ref{lem_recalc_iotamgen} we conclude that there is a graph $G'$ with $\iota_m(G')=n'$ and
\begin{equation}\label{eq_iotamproof2}
\nu(G')\le \nu(G)+\nu(\widetilde{G})\le \nu(\widetilde{G})+\gamma\cdot \log_2 k+\nu_0.
\end{equation}
By~\eqref{eq_iotamproof2} and $k\le \frac{n'}{h_{\widetilde{G}}'}$ we have
$$
\begin{array}{rl}
\nu(G')&\le \nu(\widetilde{G})+\gamma\cdot \log_2 n'-\gamma\cdot\log_2 h_{\widetilde{G}}' +\nu_0=\\
 &=\gamma\cdot \log_2 n'+\nu_0 +((\log_2 h_{\widetilde{G}}')^{-1}\nu(\widetilde{G})- \gamma)\cdot\log_2 h_{\widetilde{G}}'\le\\
 &\le \gamma\cdot \log_2 n'+\nu_0.
\end{array}
$$
\end{proof}

\begin{theorem}\label{the_iotamgen}
For all $n\in\mathbb{N}$ we have
\begin{equation}\label{eq_iotamgen}
2\log_2n\le L_{\iota_m,\nu}^{\mathcal{B}}(n)\le 2.88\log_2n + O(1).
\end{equation}
\end{theorem}

\begin{proof}
The lower bound of~\eqref{eq_iotamgen} follows from the observation that a number of m.\,i.\,s. in a bipartite graphs cannot exceed the number of subsets of any of this graph's parts.
\par To obtain the upper bound we apply lemma~\ref{lem_main_iotamgen} with $\Gamma$ being equal to the following set of graphs (subsets $U_1,\,U_2$ are marked as bold vertices; pairs of numbers $(h_{\widetilde{G}}',\,h_{\widetilde{G}}'')$ are scribed under the graphs):\\
\newcommand{\dcrc}{*[o]{\circ}}
\newcommand{\dblt}{*[o]{\bullet}}
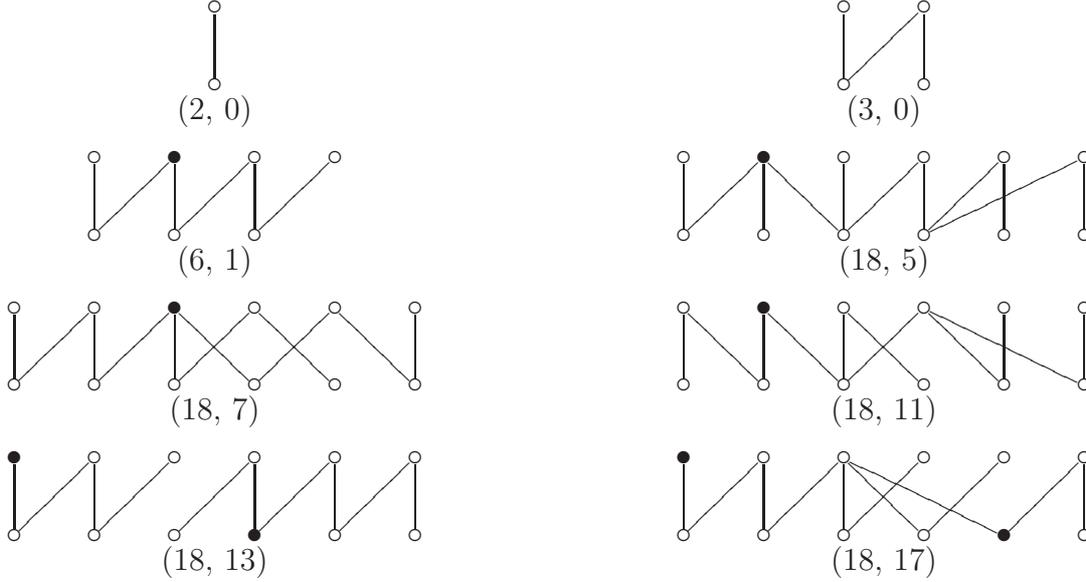
\begin{tabular}{ccc}
\xymatrix{
\dcrc\ar@{-}[d] \\
\dcrc }&\hphantom{XXXXXXXX}&
\xymatrix{
\dcrc\ar@{-}[d] & \dcrc\ar@{-}[dl]\ar@{-}[d] \\
\dcrc & \dcrc } \\

$(2,\,0)$&&$(3,\,0)$ \\

\xymatrix{
\dcrc\ar@{-}[d] & \dblt\ar@{-}[dl]\ar@{-}[d] & \dcrc\ar@{-}[dl]\ar@{-}[d] & \dcrc\ar@{-}[dl] \\
\dcrc & \dcrc & \dcrc}&&
\xymatrix{
\dcrc\ar@{-}[d] & \dblt\ar@{-}[dl]\ar@{-}[d]\ar@{-}[dr] & \dcrc\ar@{-}[d] & \dcrc\ar@{-}[dl]\ar@{-}[d] & \dcrc\ar@{-}[dl]\ar@{-}[d] & \dcrc\ar@{-}[dll]\ar@{-}[d] \\
\dcrc & \dcrc & \dcrc & \dcrc & \dcrc & \dcrc } \\

$(6,\,1)$&&$(18,\,5)$ \\

\xymatrix{
\dcrc\ar@{-}[d] & \dcrc\ar@{-}[dl]\ar@{-}[d] & \dblt\ar@{-}[dl]\ar@{-}[d]\ar@{-}[dr] & \dcrc\ar@{-}[dl]\ar@{-}[dr] & \dcrc\ar@{-}[dl]\ar@{-}[dr] & \dcrc\ar@{-}[d] \\
\dcrc & \dcrc & \dcrc & \dcrc & \dcrc & \dcrc }&&
\xymatrix{
\dcrc\ar@{-}[d]\ar@{-}[dr] & \dblt\ar@{-}[d]\ar@{-}[dr] & \dcrc\ar@{-}[d]\ar@{-}[dr] & \dcrc\ar@{-}[dl]\ar@{-}[dr]\ar@{-}[drr] & \dcrc\ar@{-}[d] & \dcrc\ar@{-}[d] \\
\dcrc & \dcrc & \dcrc & \dcrc & \dcrc & \dcrc }\\

$(18,\,7)$&&$(18,\,11)$ \\

\xymatrix{
\dblt\ar@{-}[d] & \dcrc\ar@{-}[dl]\ar@{-}[d] & \dcrc\ar@{-}[dl] & \dcrc\ar@{-}[dl]\ar@{-}[d] & \dcrc\ar@{-}[dl]\ar@{-}[d] & \dcrc\ar@{-}[dl]\ar@{-}[d] \\
\dcrc & \dcrc & \dcrc & \dblt & \dcrc & \dcrc }&&
\xymatrix{
\dblt\ar@{-}[d] & \dcrc\ar@{-}[dl]\ar@{-}[d] & \dcrc\ar@{-}[dl]\ar@{-}[d]\ar@{-}[dr]\ar@{-}[drr] & \dcrc\ar@{-}[dl] & \dcrc\ar@{-}[dl] & \dcrc\ar@{-}[dl]\ar@{-}[d] \\
\dcrc & \dcrc & \dcrc & \dcrc & \dblt & \dcrc } \\

$(18,\,13)$&&$(18,\,17)$
\end{tabular}

It can be checked, that such $\Gamma$ meets the conditions of lemma~\ref{lem_main_iotamgen} and for this set the parameter $\gamma$ would equal $12(\log_218)^{-1}<2.88$. It implies the lower bound in~\eqref{eq_iotamgen}.
\end{proof}

\begin{remark}
The inequality~\eqref{eq_iotamgen} remains valid without $O(1)$ summand, which can be proven in the same way as in theorem~\ref{the_iotamgen}. Upper bound in~\eqref{eq_iotamgen} may be directly improved by finding a better set~$\Gamma$. To find such $\Gamma$ one can apply an exhaustive computer search (which in fact was used to find $\Gamma$ that we provide above).
\end{remark}

We feel certain that the following is true:
\begin{conjecture}
$L_{\iota_m,\nu}^{\mathcal{B}}(n)\sim 2\log_2n$ for $n\to\infty$.
\end{conjecture}

Thought we were unable to prove the above conjecture, theorem~\ref{the_iotamser} approves it for some special class of naturals. Next we need to prove some auxillary statements.

\begin{lemma} \label{lem_iotamser_simple}
For any bipartite $G$ without isolated vertices there is a bipartite graph without isolated vertices having  $(\nu(G)+4)$ vertices and $(2\iota_m(G)+1)$ maximal independent sets.
\end{lemma}

\begin{proof}
Apply lemma~\ref{lem_recalc_iotamgen}, taking $P_4$ for $\widetilde{G}$, and taking any central vertex of $\widetilde{G}$ and empty set for $U_1$ and $U_2$ respectively.
\end{proof}

\begin{lemma} \label{lem_iotamser_simple2}
For any bipartite $G$ without isolated vertices there is a bipartite graph without isolated vertices having  $(\nu(G)+4)$ vertices and $(\iota_m(G)+2)$ maximal independent sets.
\end{lemma}

\begin{proof}
Apply lemma~\ref{lem_recalc_iotamgen}, taking $P_4$ for $\widetilde{G}$, and taking pair of non-adjacent vertices of $\widetilde{G}$ and an empty set for $U_1$ and $U_2$ respectively.
\end{proof}

\begin{lemma} \label{lem_iotamser_simple3}
For any bipartite graphs $G$ and $\widetilde{G}$ without isolated vertices there is a bipartite graph without isolated vetices with $(\nu(G)+\nu(\widetilde{G})+4)$ vertices and $(\iota_m(G)+\iota_m(\widetilde{G}))$ maximal independent sets.
\end{lemma}

\begin{proof}
Apply lemma~\ref{lem_recalc_iotamgen}, with whole parts of $\widetilde{G}$ selected as $U_1$ and $U_2$. Thus we obtain $G'$ on $(\nu(G)+\nu(\widetilde{G})+4)$ vertices with $(\iota_m(G)+\iota_m(\widetilde{G})-2)$ maximal independent sets. It suffices to apply lemma~\ref{lem_iotamser_simple2} to~$G'$.
\end{proof}

\begin{lemma}\label{lem_main_iotamser}
Let $G$ and $\widetilde{G}$ be bipartite without isolated vertices, and let $s,t\in{\mathbb N}$. Then there exists bipartite graph without isolated vertices having $$2^{st}\cdot\iota_m(G)+\frac{2^{st}-1}{2^t-1}\cdot\iota_m(\widetilde{G})$$ maximal independent sets and no more than $\nu(G)+\nu(\widetilde{G})+2s(t+1)+3$ vertices.
\end{lemma}

\begin{proof}
For $s=1$ the statement follows from lemma~\ref{lem_iotamser_simple3} (before applying the lemma add matching on $2t$ vertices to $G$). So for the rest of the proof we assume that $s\ge2$. We also assume that $V_G\cap V_{\widetilde{G}}=\emptyset$. Parts of $G$ and $\widetilde{G}$ will be denoted as $L_G,\,R_G$ and $L_{\widetilde{G}},\,R_{\widetilde{G}}$ respectively. We shall consider a graph $G'$ which is constructed as follows:
$$\begin{array}{rl}
V_{G'}=&V_G\cup V_{\widetilde{G}}\cup \{w\}\cup \{\widetilde{u}_i\mid 1\le i\le t\}\cup\{\widetilde{v}_i\mid 1\le i\le t\}\cup\smallskip\\
&\cup \{u_{i,j}\mid 1\le i\le s-1,\,1\le j\le t+1\}\cup \{v_{i,j}\mid 1\le i\le s-1,\,1\le j\le t+1\},\medskip\\
E_{G'}=&E_G\cup E_{\widetilde{G}}\cup \{\widetilde{u}_i\widetilde{v}_i\mid 1\le i\le t\}\cup \{u_{i,j}v_{i,j}\mid 1\le i\le s-1,\,1\le j\le t\}\cup \smallskip\\
&\cup \{uv\mid u\in L_G,\,v\in R_{\widetilde{G}}\}\cup \{uv\mid u\in R_G,\,v\in L_{\widetilde{G}}\}\cup \{wv\mid v\in R_G\cup R_{\widetilde{G}}\}\cup \smallskip\\
&\cup\{\widetilde{u}_iv\mid 1\le i\le t,\,v\in R_{\widetilde{G}}\}\cup \{u\widetilde{v}_i\mid 1\le i\le t,\,u\in L_{\widetilde{G}}\} \cup\smallskip\\
&\cup \{u_{i,t+1}v\mid 1\le i\le s-1,\,v\in R_G\}\cup \{uv_{i,t+1}\mid 1\le i\le s-1,\,u\in L_G\}\cup\smallskip\\
&\cup \{u_{i,j}v_{k,t+1}\mid 1\le i\le k\le s-1,\,1\le j\le t+1\}.
\end{array}
$$
It can be checked that $G'$ is bipartite with one of its parts being
$$L_{G'}=L_G\cup L_{\widetilde{G}}\cup \{w\}\cup \{\widetilde{u}_i\mid 1\le i\le t\}\cup \{u_{i,j}\mid 1\le i\le s-1,\,1\le j\le t+1\}.$$
We now count all maximal independent sets in $G'$. These can be of the following seven types:
\begin{enumerate}
\item Let $\mathcal{I}_1=\{I\in \mathcal{I}_m(G')\mid I\cap L_{G}\neq\emptyset,\,I\cap R_{G}\neq\emptyset \}$. Note that for every set $I\in\mathcal{I}_1$ the subset $I\cap V_G$ is m.\,i.\,s. in $G$, and that the intersection of $V_{G'}\setminus V_G$ and $I$ can only contain the following vertices: $\widetilde{u}_i$, $\widetilde{v}_i$, $u_{i,j}$ and $v_{i,j}$ for $j\neq t+1$. The subgraph generated by these vertices is a matching, which implies
    \begin{equation}\label{eq_iotam_1}
    |\mathcal{I}_1|=(\iota_m(G)-2)\cdot 2^{st}.
    \end{equation}
\item Let $\mathcal{I}_2=\{I\in \mathcal{I}_m(G')\mid I\cap L_{G}\neq\emptyset,\,I\cap R_{G}=\emptyset \}$. It can be checked that every $I\in\mathcal{I}_2$ must contain \emph{all} vertices from $L_G$, and also vertices $w$ and $u_{i,t+1}$ for all $i$. Moreover $I$ is disjoint with $R_{\widetilde{G}}$. The rest of the vertices in $I$ form a maximal independent set in subgraph, generated by the set $$L_{\widetilde{G}}\cup \{\widetilde{u}_i\mid i\le t\}\cup\{\widetilde{v}_i\mid i\le t\}\cup \{u_{i,j}\mid i\le s-1,\,j\le t\}\cup \{v_{i,j}\mid i\le s-1,\,j\le t+1\}.$$ The number of the latter is $2^{st}$, so we have
     \begin{equation}\label{eq_iotam_2}
    |\mathcal{I}_2|=2^{st}.
    \end{equation}
\item Let $\mathcal{I}_3=\{I\in \mathcal{I}_m(G')\mid I\cap L_{G}=\emptyset,\,I\cap R_{G}\neq\emptyset \}$. As in the previous case we have $|\mathcal{I}_3|=2^{st}$. With~\eqref{eq_iotam_1} and~\eqref{eq_iotam_2} it gives us
    \begin{equation}\label{eq_iotam_3}
    |\mathcal{I}_1|+|\mathcal{I}_2|+|\mathcal{I}_3|=\iota_m(G)\cdot 2^{st}.
    \end{equation}
\item We now turn to counting those m.\,i.\,s. of $G'$ that contain no vertices of $V_G$. We use the notation $\mathcal{I}_{\overline{G}}=\{I\in \mathcal{I}_m(G')\mid I\cap V_G=\emptyset\}$.
    \par Let $\hat{G}$ be a subgraph of $G$ generated by vertices $u_{i,j}$ and $v_{i,j}$, $1\le i\le s-1,\,1\le j\le t+1$. For what follows it is useful to calculate $\iota_m(\hat{G})$. The number $\hat{\iota}_0$ of m.\,i.\,s. of $\hat{G}$ which do not contain any $v_{i,t+1}$ equals to $2^{(s-1)t}$ (that is the number of m.\,i.\,s. in a matching with $(s-1)t$ edges). Next consider an arbitrary $k,\,1\le k\le s-1$. Let us count the number $\hat{\iota}_k$ of those maximal independent sets $\hat{I}$ in $\hat{G}$, that contain $v_{k,t+1}$ but do not contain any of $v_{i,t+1}$ for $i>k$. For such $\hat{I}$ we have $\hat{I}\not\ni u_{i,j}$ and $\hat{I}\ni v_{i,j}$ for all $i<k$ and for all $j$. Moreover, for such $\hat{I}$ we have $\hat{I}\ni u_{i,t+1}$ for $i>k$, and the rest of vertices in $\hat{I}$ form a maximal independent set in a matching $\{u_{i,j}v_{i,j}\mid k<i\le s-1,\,1\le j\le t\}$. From what was mentioned it follows that $\hat{\iota}_k=2^{(s-1-k)t}$. Finally we have
    \begin{equation}\label{eq_iotam_hatG}
    \iota_m(\hat{G})=\sum_{k=0}^{s-1}\hat{\iota}_k=\sum_{k=0}^{s-1}2^{(s-1-k)t}=\frac{2^{st}-1}{2^t-1}.
    \end{equation}
    \begin{enumerate}
    \item Let $\mathcal{I}_4=\{I\in \mathcal{I}_{\overline{G}}\mid I\cap L_{\widetilde{G}}\neq\emptyset,\,I\cap R_{\widetilde{G}}\neq\emptyset \}$. For $I\in\mathcal{I}_4$ the subset $I\cap V_{\widetilde{G}}$ is a m.\,i.\,s. in $\widetilde{G}$, and the set $I\cap (V_G\setminus V_{\widetilde{G}})$ is a m.\,i.\,s. in $\hat{G}$. So we get
        \begin{equation}\label{eq_iotam_4}
        |\mathcal{I}_4|=(\iota_m(\widetilde{G})-2)\cdot \iota_m(\hat{G}).
        \end{equation}
    \item Let $\mathcal{I}_5=\{I\in \mathcal{I}_{\overline{G}}\mid I\cap L_{\widetilde{G}}\neq\emptyset,\,I\cap R_{\widetilde{G}}=\emptyset\}$. Every  $I\in\mathcal{I}_5$ contains all vertices of $L_{\widetilde{G}}$ and every $\widetilde{u}_i$. Moreover, such $I$ would not contain $w$ and any $\widetilde{v}_i$. Note that $I\cap V_{\hat{G}}$ is a m.\,i.\,s. in $I\cap V_{\hat{G}}$ and should contain at least one of the vertices $v_{i,t+1}$. It implies
        \begin{equation}\label{eq_iotam_5}
        \mathcal{I}_5=\iota_m(\hat{G})-2^{(s-1)t}.
        \end{equation}
    \item Let $\mathcal{I}_6=\{I\in \mathcal{I}_{\overline{G}}\mid I\cap L_{\widetilde{G}}=\emptyset,\,I\cap R_{\widetilde{G}}\neq\emptyset\}$. Similar to the previous case we get that for every $I\in\mathcal{I}_6$ the set $I\cap V_{\hat{G}}$ is a m.\,i.\,s. in $I\cap V_{\hat{G}}$ and should contain at least one of the vertices $u_{i,t+1}$. So we have
        \begin{equation}\label{eq_iotam_6}
         \mathcal{I}_6=\iota_m(\hat{G})-1.
        \end{equation}
    \item It now suffices to find the size of $\mathcal{I}_7=\{I\in \mathcal{I}_{\overline{G}}\mid I\cap V_{\widetilde{G}}=\emptyset\}$. For every $I\in\mathcal{I}_7$ we have $w\in I$. The set $\widetilde{I}=I\cap (\{\widetilde{u_i}\mid i\le t\}\cup \{\widetilde{v_i}\mid i\le t\})$ should contain at least one of $\widetilde{v}_i$ and should me a m.\,i.\,s. in the corresponding subgraph. The number of such $\widetilde{I}$ equals to $(2^t-1)$. The set $\hat{I}=I\setminus(\{w\}\cup\widetilde{I})$ should be a m.\,i.\,s. in $\hat{G}$ and should contain at least one of $v_{i,t+1}$. The number of choices for such $\hat{I}$ is $(\iota_m(\hat{G})-2^{(s-1)t})$. At last we have
        \begin{equation}\label{eq_iotam_7}
        |\mathcal{I}_7|=(2^t-1)(\iota_m(\hat{G})-2^{(s-1)t}).
        \end{equation}
    \end{enumerate}
\end{enumerate}
By~\eqref{eq_iotam_3},~\eqref{eq_iotam_4},~\eqref{eq_iotam_5},~\eqref{eq_iotam_6},~\eqref{eq_iotam_7} and~\eqref{eq_iotam_hatG}, after some calculations we get $$\iota_m(G')=\sum_{k=1}^7|\mathcal{I}_k|=2^{st}\cdot\iota_m(G)+\frac{2^{st}-1}{2^t-1}\cdot\iota_m(\widetilde{G})-2.$$
It suffices to apply lemma~\ref{lem_iotamser_simple2} to $G'$.
\end{proof}

Let~$\overline{n}$ detone the binary representation of~$n$. Let~$w^{(k)}$ denote a binary word which consists of word $w$ repeated $k$~times.

\begin{lemma} \label{lem_iotamser_addtail}
Let $n,p,q\in{\mathbb N},\,n\ge2$. Let $n'$ be a natural number with binary representation $\overline{n}w^{(q)}$, where $w$ is a binary word of length $p$. Let $G$ be bipartite without isolated vertices, having $\iota_m(G)=n$. Then there is a bipartite graph without isolated vertices having $n'$ maximal independent sets and no more than $\nu(G)+2pq+20(p+\sqrt{pq})$ vertices.
\end{lemma}

\begin{proof}
If $w$ contains only zeros, then the desired graph can be obtained by adding a matching on $2pq$ vertices to $G$. For the rest of the proof we assume $w$ being a not-all-zero word. Firstly we consider the case $q=1$. If $w=0\ldots01$, then the desired graph is obtained from $G$ by adding a matching on $2(p-1)$ vertices and applying lemma~\ref{lem_iotamser_simple}. Otherwise, let $\widetilde{n}$ be the number with binary representation $w$ ($\widetilde{n}>1$). Then by a remark to theorem~\ref{the_iotamgen}, there exists bipartite $\widetilde{G}$ without isolated vertices having $\iota_m(\widetilde{G})=\widetilde{n}$ and $\nu(\widetilde{G})<3p$. By applying lemma~\ref{lem_main_iotamser} to $G$ and $\widetilde{G}$ with $t=p$ and $s=1$, we obtain the graph needed.
\par For the rest of the proof we assume that $q\ge 2$ and $w$ is not an all-zero word. Put $k=\max\{\lceil\sqrt{q/p}\rceil,\,2\}$, and let $r$ be the residue of $q$ modulo $k$. The remark to the theorem~\ref{the_iotamgen} implies that there exists $\widetilde{G}$ such that $\nu(\widetilde{G})<3pk$ and the binary representation of $\iota_m(\widetilde{G})$ is $w^{(k)}$ with leading zeros trimmed. The application of lemma~\ref{lem_main_iotamser} to $G$ and $\widetilde{G}$ with $t=pk$ and $s=\lfloor q/k\rfloor$ gives us a graph $G''$ with the binary representation of $\iota_m(G'')$ being $\overline{n}w^{(q-r)}$ and
\begin{equation}\label{eq_iotamser_pqk}
\nu(G'')\le \nu(G)+3pk+2(q/k)(pk+1)+3=\nu(G)+2pq+3pk+2q/k+3.
\end{equation}
This, together with the inequalities $2q/k\le 2\sqrt{pq}$ and $k\le 2+\sqrt{q/p}$ implies
$$
\nu(G'')\le \nu(G)+2pq+6p+5\sqrt{pq}+3.
$$
If $r=0$, then $G''$ is the desired graph. If $r>0$, then using the remark to theorem~\ref{the_iotamgen}, consider a graph $\widetilde{G}_r$ having $\nu(\widetilde{G}_r)\le 3r$, and the binary represenation of $\iota_m(\widetilde{G}_r)$ being equal to $w^{(r)}$ with leading zeros trimmed. Then, by lemma~\ref{lem_main_iotamser} (applied with $G''$ and $\widetilde{G}_r$ as graph $G$ and $\widetilde{G}$ respectively, $s=1$ and $t=pr$), there exists $G'$ having $\overline{\iota_m(G')}=\overline{\iota_m(G'')}w^{(pr)}=\overline{n'}$ and
$$
\nu(G')\le \nu(G'')+3r+2pr+5\le \nu(G)+2pq+6p+5\sqrt{pq}+3r+2pr+8.
$$
Using the inequality $r<k\le2+\sqrt{q/p}$, we get
$$
\nu(G')\le  \nu(G)+2pq+10p+9\sqrt{pq}+8< \nu(G)+2pq+20(p+\sqrt{pq}).
$$
\end{proof}

\begin{theorem}\label{the_iotamser}
Let $n$ be a natural number with its binary representation of the form $w_1^{(q_1)}\ldots w_k^{(q_k)}$. Let $p_i$~be the length of~$w_i$. If $\sum_{i=1}^kp_i=o(\log n)$ then the following asymptotic holds for arbitrary~$q_i$:
\begin{equation}\label{eq_iotamser_th}
L_{\iota_m,\,\nu}^{\mathcal{B}}(n)\sim 2\log_2n.
\end{equation}
\end{theorem}

\begin{proof}
The lower bound was already stated in theorem~\ref{the_iotamgen}, so we proceed to the upper. Lemma~\ref{lem_iotamser_addtail} implies that there is a graph $G$ with $\iota_m(G)=n$ and
\begin{equation}\label{eq_iotamser_th1}
\nu(G)\le 2\log_2n+O\left(\sum_{i=1}^kp_i+\sum_{i=1}^k\sqrt{p_iq_i}\right).
\end{equation}
Then the inequality $\sum_{i=1}^k p_iq_i<2\log_2n$ and Cauchy--–Bunyakovsky--–Schwarz inequality imply
\begin{equation}\label{eq_iotamser_th2}
\sum_{i=1}^k\sqrt{p_iq_i}\le\sqrt{2k\log_2n}=o(\log n).
\end{equation}
Finally~\eqref{eq_iotamser_th1} and~\eqref{eq_iotamser_th2} imply~\eqref{eq_iotamser_th}.
\end{proof}

\par The research was supported by the Russian President Grant MK-3429.2010.1 and RFBR Grant No.\,10-01-00768a.

\end{document}